\documentclass[preprint,3p]{elsarticle} 
\usepackage{amsmath}
\DeclareMathOperator{\sgn}{sgn}
\usepackage{amssymb}
\usepackage{caption} 
\usepackage{bigints}
 \usepackage[english]{babel} 
 \usepackage{float} 
 \usepackage{graphics} 
 \usepackage{graphicx} 
\usepackage[utf8]{inputenc}
\usepackage{mathtools}
\usepackage{capt-of} 
\usepackage{calrsfs}
\usepackage{lipsum}  % to provide "filler" text
 \usepackage{epsfig}
\newcommand{\R}{\mathbf{R}}

\newtheorem{theorem}{Theorem}[section]
\newtheorem{exam}{Example}[section]
\newtheorem{lemma}{Lemma}[section]
\newtheorem{cor}{Corollary}[section]

\newtheorem{prop}{Proposition}[section]

\newdefinition{remark}{Remark}[section]
\newproof{proof}{Proof}
\DeclareMathAlphabet{\pazocal}{OMS}{zplm}{m}{n}

\numberwithin{equation}{section}
\begin{document}%

\begin{frontmatter}

\title{The Converse of Sturm's Separation Theorem }
\author{Leila Gholizadeh\footnote[1]{leilagh@math.carleton.ca} and Angelo B. Mingarelli\footnote[2]{angelo@math.carleton.ca}}
\address{School of Mathematics and Statistics, Carleton University, Ottawa, Canada}

\begin{abstract}
We show that Sturm's classical separation theorem on the interlacing of the zeros of linearly independent solutions of real second order two-term ordinary differential equations necessarily fails in the presence of a turning point in the principal part of the equation.  Related results are discussed. 
\end{abstract}

\end{frontmatter}

\section{Introduction}
In the sequel we will always assume, unless otherwise stated that 
\begin{equation}
\label{eq00}\frac{1}{p}, q \in L(I),\quad [a,b] \subset I
\end{equation}
where $I$ is a closed and bounded interval and the functions $p, q : I \to \R$. In this paper there are generally no sign restrictions on the principal part of \eqref{eq01}, i.e., the values, $p(x)$, are generally unrestricted as to their sign and $p(x)$ may even be infinite on sets of positive measure. As usual the symbol $||*||_1$ will denote the $L(I)$ norm.

It is well known \cite{er} that the conditions \eqref{eq00} imply the existence and uniqueness of Carath\'{e}odory solutions of initial value problems associated with \eqref{eq01}, \begin{equation}
\label{eq01}
-(p(x)y^\prime)^\prime + q(x)\,y = 0,\quad\quad x \in [a,b],
\end{equation} that is, solutions $y$ such that both $y$ and $py^\prime$ are absolutely continuous on $[a,b]$ and satisfy
\begin{equation}
\label{eq02}
y(a) = y_a,\quad py^\prime(a) = y_{a^\prime},
\end{equation}
for given $y_a, y_{a^\prime}$. The study of problems with an indefinite leading term (a.k.a. an {\it indefinite principal part}) are few and far between. For example, the failure of Sturm's oscillation theorem in such indefinite cases was observed in [\cite{atm}, p.381] where, in the presence of an indefinite weight function, it may occur that the spectrum is, in fact, the whole complex plane and the eigenfunctions behave in a totally non-Sturmian fashion. The example in question consists in choosing $p(x)=q(x)=\sgn x$ for $x\in [-1,1]$, $y(-1)=y(1)=0$. Then the two solutions $y_1(x) = \sin P(x)$ and $y_2(x)=\cos P(x)$ where $P(x) = |x|-1$ have non interlacing zeros. Indeed, $y_2(x)\neq 0$ on $[-1,1]$ while $y_1(x)$ vanishes at both ends there. This special case is contained in Theorem~\ref{th0} below.  

Recall that, in its simplest most classical form, Sturm's Separation Theorem states that given any non-trivial solution $y$ of \eqref{eq01} having consecutive zeros at $a, b$, $a<b$, where $[a,b]\subset I$ then every other linearly independent solution of \eqref{eq01} must vanish only once in $(a,b)$. An equation \eqref{eq01} is said to have the {\it Sturm Separation Property} (abbr. SSP) on $[a,b]$ if Sturm's Separation Theorem holds for the given equation on the given interval.

The framework described above normally assumes that the principal part, $p$, appearing in \eqref{eq01} is a.e. finite on $[a,b]$. However, still greater generality can be obtained by allowing $p(x)$ to be identically infinite on subintervals. In this case one needs to rewrite \eqref{eq01} as a vector system in two dimensions, e.g., \begin{equation}\label{eq06} u^{\prime} =  \frac{v}{p},\quad  v^{\prime} = q\, u.\end{equation}
This now defines a problem of {\it Atkinson-type} (see [\cite{atk}, Chapter 8], [\cite{jb}, p.558] for more details). The advantage of using this  formulation is that it can be used to study three-term recurrence relations as well, see \cite{atk}, \cite{abm}.  We summarize this approach briefly: We divide $[a,b]$ into a finite union of subintervals 
\begin{equation}
\label{eqpar}
[a,b_o],\, [b_o,a_1],\, [a_1,b_1],\, [b_1,a_2],\,[a_2,b_2],\, ...\, [b_{m-1},a_m],\,[a_m,b].
\end{equation}
on each of which alternately $p(x) = \infty$ or $q(x)=0$ (but $p(x)$ is not infinite when $q(x)=0$). Direct integration of \eqref{eq06} then shows that $y_n = u(a_n)$ satisfies the three-term recurrence relation
\begin{equation}\label{eq07}
c_n\,y_{n+1}+c_{n-1}\,y_{n-1}-d_n\,y_n = 0,
\end{equation}
where 
$$ c_n^{-1} = \int_{b_n}^{a_{n+1}}\frac{ds}{p(s)},\quad d_n = c_n + c_{n+1} + \int_{a_n}^{b_n} q(s)\, ds,$$
or, equivalently, a second order difference equation
\begin{equation}\label{rr}
 - \triangle (c_{n-1}\triangle y_{n-1}) + (\int_{a_n}^{b_n} q(s)\, ds)\, y_n = 0,
\end{equation}
where, as usual, $\triangle$ represents the forward difference operator $\triangle y_n = y_{n+1}-y_n$.

Recall that by a {\it zero} of a solution of \eqref{eq07} is meant the zero of that absolutely continuous polygonal curve with vertices at $(n,y_{n})$. (This interpretation arises directly by integrating \eqref{eq06}.) Thus, zeros of solutions of \eqref{eq07} are said to {\it interlace} if the corresponding polygonal curves have interlacing zeros.

The failure of Sturm's Separation Theorem (or SSP) in the case of recurrence relations (or difference equations) is old but chronicled by both B\^{o}cher \cite{mb} and Moulton, \cite{ejm}, and not independently of one another. (Moulton \cite{ejm} actually cites B\^{o}cher in reference to the question.) B\^{o}cher \cite{mb} goes on to give, as an example, two independent solutions of the Fibonacci sequence recurrence relation,
$$y_{n+1} = y_n + y_{n-1},\quad y_{-1}=0, y_0=1;\quad y_{-1}= -10, y_0=6,$$
with no interlacing features whereas Moulton \cite{ejm} went on to show (at B\^{o}cher's prodding) that \eqref{eq07} has the SSP provided $c_n\,c_{n-1} > 0$ for all $n$. To the best of our knowledge, a converse of Sturm's Separation Theorem has not been addressed. In [\cite{abm}, p.209] we showed, by means of an example, that the SSP may fail in the case where Moulton's condition $c_n\,c_{n-1} > 0$ fails. 

%%%%%%%%%%

This failure suggests that $p(x)$ must change its sign in the continuous case and that intervals in which violations to SSP occur must be neighborhoods of a ``turning point" of $p$. The existence of such a point is necessitated by the fact that otherwise $p(x)$ would be (a.e.) of one sign in $[a,b]$ and so SSP must hold there by Sturmian arguments. 

Below we present a converse to SSP as a consequence of more general results dealing with \eqref{eq06}. Said result will then apply to both differential and difference equations. 

Specifically, we will prove (Theorem~\ref{th0}) that whenever the leading term $p(x)$ has a turning point in $(a,b)$ then SSP must fail. This is equivalent to showing that if the SSP holds then $p(x)$ cannot have a turning point inside $(a,b)$ and thus $p(x)$ is a.e. of one sign. This is the actual converse of Sturm's Separation Theorem. We illustrate this result by means of explicit examples.

We also present (Theorem~\ref{th3}) an effective necessary condition for the existence of a solution vanishing at the end-points of an interval in the case of sign-indefinite $p$ and $q$. Examples are provided illustrating the various theorems. Other results of independent interest are also presented thus demonstrating the complexities of qualitative behavior of solutions in the case of indefinite leading terms.  

We conclude by a conjecture which gives an upper and lower bound to the difference in the number of zeros in $[a,b]$ between two independent solutions in the case of an arbitrary but finite number of turning points in $p(x)$.

\section{Main results}
We recall that if $p$ is continuous or piecewise continuous on $[a,b]$ then a {\bf turning point} is a point $c\in (a,b)$ around which $p(x)$ changes its sign. If $p$ is merely measurable then $c$ is defined by requiring that, in  some interval containing $c$ in its interior, we have $(x-c)p(x)>0$ a.e. (or $(x-c)p(x)<0$ a.e.) This somewhat restrictive definition implies that the set of turning points of $p$ cannot be everywhere dense in $(a, b)$. Indeed, this definition implies that turning points must be separated from one another.

In the sequel we always assume that solutions of \eqref{eq06} or \eqref{eq06i} below are deemed non-trivial. In addition, we take it that $1/p(x)$ may vanish a.e. on sets of positive measure, but not vanish a.e. on $[a,b]$, and that $p(x)$ is unrestricted as to its sign there.

\begin{lemma}
\label{lem0}
For $i=1,2$, let $u_1,u_2$ be solutions of 
\begin{equation}\label{eq06i} u_i^{\prime} =  \frac{v_i}{p},\quad  v_i^{\prime} = q\, u_i,\end{equation}
where $p, q$ satisfy \eqref{eq00}. Then 
\begin{equation}\label{eq0001}
u_2(x)v_1(x)-u_1(x)v_2(x) = C,
\end{equation}
where $C$ is a constant. 
\end{lemma}

We will assume that, without loss of generality, $C=1$. The main result shows that SSP fails whenever $p$ has a turning point and thus the a.e. positivity (or negativity) of $p(x)$ is a necessary condition for the validity of SSP as well as sufficient, as is well known.

\begin{theorem}
\label{th0}
Let $p(x)$ have a unique turning point at $x=c$, $a < c < b$  and let $u_i$, $i=1,2$ be linearly independent solutions of \eqref{eq06i} such that $u_1(a)=u_1(b)=0$, $u_1(x) \neq 0$ in $(a,b)$. Then either $u_2(x) \neq 0$ on $[a,b]$, or  $u_2(x)$ is of constant sign except only at $x=c$ where $u_2(c)=0$, or finally $u_2(x)$ has exactly two zeros in $(a,b)$. In every case it follows that SSP fails on $[a,b]$.
\end{theorem}

\begin{remark}\label{rem21}
The previous theorem is independent of the sign of $q(x)$ and the number of turning points of $p$ and assumes only that a solution exists vanishing at two points around a given turning point. This is the general case as otherwise the existence of two consecutive zeros in a turning-point-free set would lead to SSP there by classical Sturm theory since $p(x)$ is a.e. of one sign.

The result also includes an analog for the difference equation \eqref{rr} above. Basically, if the $c_n$ change sign at least once, then the solutions, viewed as polygonal curves, have the property stated in the theorem. 

The next example illustrates the result in the continuous case.
\end{remark}
\begin{exam}\label{exam21}
Let $I=[0,\pi]$, and consider the differential equation 
\begin{equation*}
u^{\prime} =  \cos(x)\, v,\quad  v^{\prime} = -  \cos(x)\, u.
\end{equation*}
with a unique turning point at $c=\pi/2$. Then the general solution is 
$$y(x) = c_1 \sin (\sin x)  + c_2 \cos (\sin(x)),$$
where $c_1, c_2$ are constants. First, note that solution $u_1(x) = \sin(\sin x)$ satisfies the conditions of Theorem~\ref{th0}. We now  exhibit solutions of the type guaranteed by said theorem.
\begin{itemize}
\item The solution $u_2(x) = \cos(\sin x)$ has no zeros in $[0,\pi]$.  
\item The solution $u_2(x) = - \cos 1\, \sin(\sin x) + \sin 1\,\cos(\sin (x)) \geq 0$ on $[0, \pi]$ and it has exactly one zero at the turning point $x=\pi/2$ bouncing positively there.
\item The solution $u_2(x) = \cos(\sin x) - \sin(\sin x)$ has exactly two zeros, in conformity with said theorem. 
\item Every solution of this equation has at most two zeros. 
\end{itemize}
The latter result is most readily proved by contradiction. Assuming three such zeros $x_i, i=1,2,3$, $x_i \in [0, \pi]$, we can easily deduce that the three quantities $\tan(\sin(x_i))$ have a common value (i.e., independent of $i$) and this is impossible on $[0,\pi]$.
As a result, SSP fails for this equation.
\end{exam}
Next we consider the problem of finding necessary and sufficient  conditions for the existence of two zeros of \eqref{eq06} on $[a,b]$, i.e., in particular, we are asking for conditions under which  this equation not disconjugate. For the notion of disconjugacy we refer the reader to \cite{jb}, \cite{ph}. 
\begin{theorem}
\label{c21}
The equation \eqref{eq06} with $q(x)=0$ a.e. on $[a,b]$ has a non-trivial solution satisfying $u(a) = u(b)=0$ if and only if \begin{equation}\label{eq05}
 \int_a^b \frac{ds}{p(s)} = 0,
\end{equation}
\end{theorem}
\begin{cor}
\label{c22}
Let $c_n$ satisfy 
\begin{equation}\label{eq08}\sum_{n=0}^{m-1} c_n^{-1} = 0.
\end{equation}
Then SSP fails for three-term recurrence relations of the form 
\begin{equation}\label{eq09}
c_n\,y_{n+1}+c_{n-1}\,y_{n-1}-(c_n + c_{n-1})\,y_n = 0.
\end{equation}
\end{cor}

\begin{exam} Let $c_{n-1} = (-1)^{n}$, $n=0,\ldots,m$, where $m$ is even. Then \eqref{eq09} reduces to $y_{n+1} = y_{n-1}$. This has two linearly independent solutions defined by the initial conditions, $y_{-1}=0$, $y_0=1$ and $y_{-1}=1$, $y_0=2$ the former of which has numerous zeros while the second has none. We can see that SSP fails both by direct computation and by Corollary~\ref{c22}.

On the other hand, the same initial conditions $y_{-1}=0$, $y_0=1$ and $y_{-1}=1$, $y_0=2$ for the slightly modified  recurrence relation $y_{n+1} =  - y_{n-1}$ gives two solutions satisfying SSP by Moulton's theorem, \cite{ejm}.
\end{exam}
The separation property for the zeros of the {\it quasi-derivatives} of solutions, i.e., terms of the form $(py^\prime)(x)$, is next. Although the result is simply proved we have been unable to find a reference to it and so present it here for the sake of completeness.
\begin{prop}
\label{th2}
For $p, q$ as in \eqref{eq00}, let $p(x)$ be sign indefinite. In addition, let $q(x)$ be a.e. of one sign on $[a,b]$ and let $y$ be a non-trivial solution of  \eqref{eq01} satisfying 
\begin{equation}
\label{eq11}
(py^\prime)(a) = 0 = (py^\prime)(b).
\end{equation} Then for any linearly independent solution $y_1$ of \eqref{eq01} there is exactly one point $c \in (a,b)$  such that \\ $(py_1^\prime)(c) = 0$.
\end{prop}

\begin{remark}\label{rem23}
This proposition seems to be the closest that one can get to a SSP-type result for positive $q$. In other words, as we have seen earlier, the SSP fails even if $q(x)>0$ on $[a,b]$, and $p(x)$ is sign indefinite (i.e., has a turning point in $(a,b)$).  
\end{remark}

Next, we give a necessary condition for the existence of a solution vanishing at the end-points of a typical interval, $[a,b]$, and positive in its interior in the presence of an indefinite principal part or {\it leading term}, $p(x)$, in \eqref{eq06}.
\begin{theorem}
\label{th3}
Let $||q||_1>0$ and let \eqref{eq05} hold. Let $u$ be a solution of \eqref{eq06} such that $u(a)=u(b)=0$, and $u(x) > 0$ for $x \in (a,b)$. Then, writing,
\begin{equation}
\label{eq12x}
P(x) := \int_a^x \frac{ds}{p(s)},
\end{equation}
either $P(x)q(x)=0 \quad {\text{a.e. on $(a,b)$}}$ or there is a set of positive measure on which $P(x)q(x) > 0$ a.e. in $(a,b)$ and  a set of positive measure on which $P(x)q(x) < 0$ a.e. in $(a,b)$ (i.e., $Pq$ changes its ``sign" on $(a,b)$.)

\end{theorem}
The next result is of independent interest, Example~\ref{exam21} being a special case.
\begin{lemma}
\label{lem2}
Let $I=[a,b]$, $\lambda > 0$. The general solution of either 
\begin{equation*}
(py^\prime)^\prime + \frac{\lambda}{p}y=0, \quad\quad {\text or}\quad\quad  u^{\prime} =  \frac{v}{p},\quad  v^{\prime} = - \frac{\lambda}{p}\, u.
\end{equation*}
is given by 
$$y(x) = u(x) = c_1 \cos \left (\sqrt{\lambda}P(x)\right ) + c_2 \sin \left (\sqrt{\lambda}P(x)\right ),$$ where
$$ v(x) =   - c_1 \sqrt{\lambda}\, \sin \left (\sqrt{\lambda}P(x)\right ) + c_2 \sqrt{\lambda}\, \cos \left (\sqrt{\lambda}P(x)\right ),$$
where $c_1, c_2$ are constants.
\end{lemma}
\begin{remark}
It is well known and easy to derive that in the case where the leading term $p(x)$ is a.e. positive (or negative) then the absolute value of the difference of the number of zeros of two independent solutions is equal to 1, due to the interlacing property of such zeros. In the case of an indefinite leading term we make the following conjecture.

\section{Conjecture}

Let $p(x)$ have at least one turning point in $(a,b)$ and let $y$ be a solution satisfying $y(a)=y(b)=0$ having $n$ zeros in $[a,b]$. Then, given any integer $k$, $0 \leq k \leq n$, there are examples for which the absolute value of the difference of the number of zeros of two independent solutions on $[a,b]$ is equal to $k$. 

This totally non-Sturmian behavior appears to be typical in cases where the principal part changes sign.
\end{remark}

%%%%%%%%%%%%%%%%%%%%%%%%%%
\section{Proofs}

\begin{proof} (Lemma~\ref{lem0}) The proof is by differentiation of the expression on the left of \eqref{eq0001} making use of \eqref{eq06i}. Note that all $u_i, v_i$, and so their products, are absolutely continuous on the interval under consideration.
\end{proof}

\begin{proof} (Theorem~\ref{th0}) There are only two logical possibilities. Either $u_2(x) \neq 0$  in $[a,b]$ or $u_2(x) = 0$ at $x=x_0$ in $(a,b]$. Clearly $u_2(a) \neq 0$ as its negation would violate \eqref{eq0001}. For the sake of simplicity we may assume that $u_2(a)>0$ (or else we may replace $u_2$ by $-u_2$ in the ensuing discussion along with other minor changes). 

In addition, we may assume, without loss of generality, that this first zero is at, say $x_0 \in (a,c)$, that is, to the left of the turning point. A similar argument applies in the event that this zero is in $(c,b]$. Thus, $u_2(x) \geq 0$ for $x \in [a,x_0)$.

Next, we show that, unless $x_0=c$ (see below), $u_2(x)$ cannot ``bounce" off $x=x_0$ and remain positive for some $x>x_0$. To see this observe that  \eqref{eq0001} implies that $v_2(x_0)<0$. The continuity of $v_2$ now implies the existence of a $\delta > 0$ and a neighborhood $J = (x_0-\delta, x_0+\delta) \in (a,c)$ in which $v_2(x)<0$. It follows that, for $x\in (x_0, x_0+\delta)$,
$$u_2(x)  = \int_{x_0}^{x} \frac{v_2(s)}{p(s)}\, ds.$$
Since $p(x) > 0$ a.e. in $J$ and $v_2(x) <0$ there as well, we see that $u_2(x) < 0$ to the right of $x_0$ and thus $u_2$ must cross the axis whenever it is zero. Summarizing, we have shown that there exists a $\delta >0$ such that $u_2(x) >0$ on $[a,x_0)$ and $u_2(x) < 0$ on $(x_0, x_0+\delta)$. Now, since $p(x) >0$ a.e. in $[a, c]$, by ordinary Sturm theory we get that it is impossible for $u_2(x)=0$ again in $(x_0+\delta, c]$. This is because SSP applies on intervals in which $p(x)$ is a.e. of one sign, and so $u_2(x)$ can have at most one zero there. It follows that $u_2(c)<0$.

As before we know that \eqref{eq0001} forces $u_2(b) \neq 0$. We show that $u_2(b)>0$. Assume the contrary, i.e., $u_2(b) < 0$. Since $p(x) < 0$ a.e. on $(c,b)$ we have from \eqref{eq0001} that $u_2(b)v_1(b)=1$ and so that $v_1(b)<0$. A continuity argument again implies the existence of a $\eta > 0$ such that  $v_1(x)<0$ for $x \in (b-\eta, b)$. For such $x$,
$$u_1(b) - u_1(x)  = -u_1(x) = \int_{x}^{b} \frac{v_1(s)}{p(s)}\, ds.$$
However, $p(x)<0$ a.e. in $(b-\eta, b)$. Hence $u_1(x) <0$ in $(b-\eta, b)$ and this contradicts the fact that $u_1(x)>0$ on $(a,b)$. Hence $u_2(b)\geq 0$. As before, the case $u_2(b)= 0$ being excluded by \eqref{eq0001}, we find that $u_2(b) > 0$.  Since $u_2$ is continuous and $u_2(c)<0$ there must exist another zero $x_1 \in (c,b)$. This zero must be unique by Sturm theory since $p(x)$ is a.e. of one sign on $(c,b)$, i.e., SSP applies here. 

Finally, let's consider the case where $x_0=c$, that is, the first zero of $u_2$ occurs at the turning point itself. This case may occur and  a bounce is possible here. The reason for this is that previous argument fails on account that $p(x)$ a.e. changes its sign on every interval of the form $(c-\delta, c+\delta)$, by definition. Since $p(x) < 0$ a.e. on $(c, c+\delta)$, and arguing as above, we get that for all $x \in (c, c+\delta)$ and $\delta$ sufficiently small, 
$$u_2(x)  = \int_{c}^{x} \frac{v_2(s)}{p(s)}\, ds > 0.$$
Thus, a bounce may occur there. Finally, $u_2(x)$ may not vanish again in $(c, b)$ since $p(x)$ is a.e. of one sign and so can only have at most one zero in $[c,b]$ by Sturm theory. This completes the proof. 
\end{proof}

%%%%%%%%%%%%%%%%

\begin{proof}(Corollary~\ref{c22}) This follows from the discussion leading to the recurrence relations.
\end{proof}

\begin{proof}(Proposition~\ref{th2})
We use the so-called {\it reciprocal transformation} \cite{jb}: Let $z=py^{\prime}$ where $y$ satisfies \eqref{eq01}. Then $z$ satisfies the equation
$$-(\frac{1}{q}\,z^\prime)^\prime + \frac{1}{p}\,z = 0, $$
and 
$$z(a) = z(b) = 0.$$
Since $q$ is a.e. of one sign, classical Sturmian results apply so that the previous equation has the SSP on said interval. Thus, for any other linearly independent solution $z_1(x)$ there is a unique $c\in (a,b)$ such that $z_1(c)=0$. In particular, if we define a solution $y_1$ via $z_1=py_1^\prime$, then $z_1(c)=0$ for some $c$, and the result follows.
\end{proof}

\begin{proof}(Theorem~\ref{th3}). 
Without loss of generality we can assume that $u(a)=0$, $v(a)=M$ where $M\neq 0$ is arbitrary but fixed. Then
\begin{eqnarray*}
u(x) &=& M\int_a^x \frac{ds}{p(s)} + \int_a^x \frac{1}{p(s)} \int_a^s q(t)u(t)\,dt\,ds\\
&=& M\int_a^x  \frac{ds}{p(s)} + P(x)\int_a^x q(t)u(t)\, dt - \int_a^x  P(t) q(t)u(t)\,dt.
\end{eqnarray*} 
Since $u(b)=0$ and $P(b)=0$, it follows that $$ \int_a^b  P(t) q(t)u(t)\,dt = 0,$$
and, since $u(x) > 0$ in $(a,b)$, the result follows.
\end{proof}

\begin{proof} (Lemma~\ref{lem2})
This is a direct calculation and so the proof is omitted. 
\end{proof}

\section{Note added in proof, Sept. 14, 2021}

This section is new and does not appear in the paper \cite{gm} on which all the previous sections are based.

In the previous discussion we showed that whenever the leading term $p(x)$ has a unique turning point in $(a,b)$ then SSP must fail. This is equivalent to showing that if the SSP holds then $p(x)$ cannot have a turning point inside $(a,b)$ and thus $p(x)$ is a.e. of one sign. This is the actual converse of Sturm's Separation Theorem. In this brief note we show that said result is true in the case where $p$ is piecewise continuous on $[a,b]$ (and so has a finite number of turning points).

In the sequel all solutions of \eqref{eq06} are assumed to be non-trivial unless otherwise specified. The first lemma is of importance as it shows that zeros of solutions, $u$, of \eqref{eq06} cannot accumulate inside $(a,b)$.

\begin{lemma}\label{lem1x}
The zeros of any solution of \eqref{eq06} cannot accumulate on $(a.b)$.  
\end{lemma}
\begin{proof}(Lemma~\ref{lem1x}) Assume the contrary. Thus there is a sequence of points $x_n  \neq x_0$ such that $u(x_n)=0$ and a limit point $x_0$ such that $x_n \to x_0$. Since $u$ is continuous it follows that $u(x_0)=0$. On the other hand, the first of \eqref{eq06} implies that
$$\frac{1}{x_n-x_0}\int_{x_n}^{x_0} \frac{v(t)}{p(t)}\, dt = \frac{u(x_0) - u(x_n) }{x_n - x_0} = 0,$$
which, in turn implies that, for all $n$,
$$\frac{1}{h_n} \int_{h_n}^{h_n + x_0} \frac{v(t)}{p(t)}\, dt = 0.$$
However, from [\cite{ect}, Section 11.6]  we know that $$\lim_{h_n\to 0} \frac{1}{h_n} \int_{h_n}^{h_n + x_0} \frac{v(t)}{p(t)}\, dt = \frac{v(x_0)}{p(x_0)}.$$
So, $v(x_0)=0$ is necessary and this, when combined with $u(x_0)=0$ implies that the solution of \eqref{eq06} under consideration is the trivial solution, a contradiction.
\end{proof}

The next lemma is also fundamental yet its statement and proof rarely appears in texts on the subject.  
\begin{lemma}\label{lem1}
If $p(x)$ is a.e. of one sign on $(a,b)$ and for some $x_0\in (a,b)$ we have $u(x_0)=0$ then $u(x)$ must change sign around $x_0$.
\end{lemma}
\begin{proof}(Lemma~\ref{lem1})
By the fundamental existence and uniqueness theorem for Carath\'{e}dory solutions of \eqref{eq06} (see \cite{er}) we know that $v$ exists and is continuous in $(a,b)$. Without loss of generality we can assume that $p(x)>0$ a.e. or else we can change $p, q$ to $-p, -q$. We can also assume, without loss of generality, that $u(x) >0$ in a left-neighborhood of $x=x_0$ (or else replace $u$ by $-u$ in what follows). Now assume, on the contrary, that $u(x)$ bounces off the $x$-axis, that is $u(x) \geq 0$, in some neighborhood of $x_0$. Then $x=x_0$ is a local minimum of $u$. For $\delta >0$ sufficiently small and for every $x\neq x_0$, in $(x_0-\delta, x_0+\delta)$, we must have  $u(x) > 0$. Now, in the latter interval, since $u$ is absolutely continuous, there is a sequence $x_n\to x_0^-$ on which $u^{\prime}(x_n)$ exists and $u^{\prime}(x_n) \leq 0$. Thus, $v(x_n) \leq 0$. Since $v$ is necessarily continuous it follows that $v(x_0)\leq 0$. An analogous argument shows that there is a sequence $x_n^*\to x_0^+$ on which $u^{\prime}(x_n^*) \geq 0$. Thus, $v(x_n^*) \geq 0$. This leads to $v(x_0)\geq 0$. Combining these we obtain that $v(x_0)=0$. Since $u(x_0)=0$ as well, it now follows that $(u(x), v(x)) \equiv 0$ by the existence and uniqueness of Carath\'{e}dory solutions, which is a contradiction. 
\end{proof}

In other words, $u(x)$ can never vanish and stay of one sign around that zero. Thus, in the {\it Sturmian case}, in order for SSP to hold it is necessary that a solution changes it sign around any of its zeros. Thus, if a solution $u$ `` bounces" off the $x$-axis, then SSP must fail.

We now show that the case of finitely many turning points may be proved using the technique in the one turning point case of Theorem~\ref{th0} above, although we require that $p$ be piecewise continuous. 
\begin{theorem}
\label{th00}
Let $p$ be piecewise continuous and have $n \geq 2$ turning points on $(a,b)$. Let $u_1$ be a solution of \eqref{eq06i} such that $u_1(a)=u_1(b)=0$, $u_1(x) \neq 0$ in $(a,b)$. Then SSP fails on $[a,b]$.
\end{theorem}

\begin{proof}(Theorem~\ref{th00}).
We label the turning points $a<c_1< c_2 < \ldots < c_n<b$.  Fix $x_0 \in (c_{n-1}, c_n)$. Then, by the fundamental existence and uniqueness theorem \cite{er},  there is a non-trivial solution $u_2$ of \eqref{eq06} such that $u_2(x_0)=0$ and $v_2(x_0)=1$. Since $p$ is piecewise continuous, we may assume without loss of generality, that $p(x)>0$ in $(c_{n-1}, c_n)$ (or else replace $u_2$ by $-u_2$ in the ensuing argument with obvious modifications). We can now apply the proof of Theorem~\ref{th0} on $[c_{n-1}, b]$, note that $x_0\neq c_n$, by construction, and then deduce that $u_2$ cannot bounce at $c_n$ so that therefore $u_2(x)$ must be zero again in $(c_n, b)$. This shows that SSP must fail in $(c_{n-1}, b)$ and so, a fortiori, in $[a,b]$.
\end{proof}

\begin{center} ACKNOWLEDGMENT
\end{center}
The authors should like to thank Mr. Seyifunmi Ayeni, of Carleton University, who provided the basis for some of the examples considered here during a summer research project at Carleton University under the supervision of the second author. We must also thank the referee for a careful reading of the manuscript.
%\clearpage

%%%%%%%%%%%%%Bibliography follows%%%%%%%%%

\end{document}